\documentclass[11pt]{article}

\usepackage{epsf,latexsym,amsfonts,amsbsy}
\usepackage{abstract,amsmath, amssymb, amsthm,amscd,amsxtra}
\usepackage{graphicx}
\usepackage{yhmath}
\usepackage{tikz}
\usetikzlibrary{calc}
\usetikzlibrary{intersections,decorations.markings}
\usepackage{booktabs}
\usepackage{threeparttable}
\usepackage{amsmath}
\usepackage[section]{placeins}

\newtheorem{theorem}{Theorem}[section]
\newtheorem{assumption}[theorem]{Assumption}
\newtheorem{definition}[theorem]{Definition}

\newtheorem{corollary}[theorem]{Corollary}

\newtheorem{lemma}[theorem]{Lemma}

\theoremstyle{definition}
\newtheorem{example}{Example}[section]
\newtheorem{remark}[theorem]{Remark}

\newcommand{\range}{\mbox{\rm Range}}
\newcommand{\Null}{\mbox{\rm Null}}
\newcommand{\tr}{\mbox{\rm tr}}
\newcommand{\rank}{\mbox{\rm rank}}

\newcommand{\re}{\mbox{\rm Re}}
\newcommand{\im}{\mbox{\rm Im}}

\title{On the tightness of an SDP relaxation for homogeneous QCQP with three real or four complex homogeneous constraints\thanks{This work was supported by the National Natural Science Foundation of China (Grants Nos. 11871115,12171052,11971073, 12171051). This work was also supported by Beijing Natural Science Foundation (Grants No. Z220004).}}

\author{Wenbao Ai$^\dag$, Wei Liang$^\dag$, Jianhua Yuan \thanks{School of Science, Beijing University of Posts and Telecommunications, Beijing 100876, China 
({aiwb@bupt.edu.cn, liangwei@bupt.edu.cn, jianhuayuan@bupt.edu.cn}). Corresponding author's Email: aiwb@bupt.edu.cn}}
\date{}

\begin{document}
\maketitle

\begin{abstract}
	In this paper, we consider the  problem of minimizing a general homogeneous quadratic function, subject to three real or four complex homogeneous quadratic inequality or equality constraints. For this problem, we present a sufficient and  necessary test condition to detect whether its typical semidefinite programming (SDP) relaxation is tight or not. This test condition is easily verifiable, and is based on only an optimal solution pair of the SDP relaxation and its dual. When the tightness is confirmed, a global optimal solution of the original problem is found simultaneously in polynomial-time.  Furthermore, as an application of the test condition, S-lemma and Yuan's lemma are generalized to three real and four complex quadratic forms first under certain exact conditions, which improves some classical results in literature. Finally, numerical experiments demonstrate  the numerical effectiveness of the test condition.
\end{abstract}



\section{Introduction}
In this paper we consider the following homogeneous quadratically constrained quadratic programming (HQCQP) over the real or complex number field (denoted by $\mathcal{F}$, $\mathcal{F}=\mathcal{R}$ or $\mathcal{C}$):
\begin{equation}\label{eq:basic-model}
	\begin{array}{lll}
		(HQP_{m})&\mbox{minimize}   &x^HA_0x\\
		&\mbox{subject to} &x^HA_ix\unlhd_i\,c_i,\ i=1,2,\cdots,m,
	\end{array}
\end{equation}
where $x\in\mathcal{F}^n$; $A_0,\,A_1,\cdots,A_m\in\mathcal{HF}^{n}$ are $n\times n$ Hermitian matrices over the number field $\mathcal{F}$;  $(c_1,c_2,\cdots, c_m)^T\in\mathcal{R}^m$ and $c_m\neq 0$; $\,\,\unlhd_i\in\{\leq,=\}$ ($i=1,2,\cdots,m$), that is, each constraint can freely be  either an inequality or an equality. 

The model $(HQP_{m})$ is very inclusive and has wide application values. Firstly, any nonhomogeneous QCQP problem with $m-1$ constraints can turn into one homogeneous QCQP problem with $m$ homogeneous constraints described by $(HQP_{m})$. Secondly, many optimization problems in Engineering are just homogeneous QCQP problems as in \eqref{eq:basic-model}.  In general, the model $(HQP_{m})$ is a NP-hard problem. So this paper will focus on the following two cases: $m=3$ for real-valued and $m=4$ for complex-valued.  Even for the two cases, due to none of the matrices $A_i$ ($i=0,1,,\cdots,m$) is restricted to being positive semidefinite, when and how to obtain a global optimal solution of $(HQP_{m})$ from its SDP relaxation is still a pending issue.



If $\mathcal{F}=\mathcal{R}$ and $m\leq2$, it was confirmed long before that the SDP relaxation of  $(HQP_{m})$ is tight. The corresponding research can be traced back to S-lemma of Yakubovich \cite{Yakubo1971} and matrix rank-one decomposition procedure of Sturm and Zhang \cite{Sturm2003On}, which show that any quadratic programming with one quadratic inequality constraint (QIC1QP) has strong duality and has no optimality gap with its SDP relaxation. In 2016, Xia, Wang and Sheu\cite{Xia2016} extended Finsler's lemma to two nonhomogeneous quadratic functions, which reveals first that a quadratic programming with one quadratic equality constraint (QEC1QP) has conditionally strong duality. If $\mathcal{F}=\mathcal{R}$ and $m=3$, the solution situation of $(HQP_{m})$ becomes far more complicated than that at $m=2$. So researchers have  first studied a special QCQP problem called  Celis-Dennis-Tapia (CDT) subproblem, which is an extended trust region subproblem with a ball constraint and a general quadratic inequality constraint    (\cite{AiZhang2009,Beck2009,Celis1985,Chen2000,Chen2001,Yuan2017new,Yuan1990,Yuan1991,Ye2003New,Zhang1992}). Among them there is one result closely related to the current paper, which is a sufficient and necessary test condition  presented by Ai and Zhang \cite{AiZhang2009} in 2009, to detect whether the SDP relaxation of the CDT subproblem is tight or not. 
In 2021, the test condition was improved by Cheng and Martins \cite{Cheng2021}, so as to be made available for two general quadratic inequality constraints. In addition,  Nguyen, Nguyen and Sheu \cite{Nguyen2019} have extended the test condition for the homogeneous QCQP with two homogeneous quadratic inequality constraints and one unit sphere constraint. 

When $\mathcal{F}=\mathcal{C}$, due to the complex S-lemma of Fradkov and Yakubovich \cite{Fradkov1973} and the complex matrix rank-one decomposition procedure of Huang and Zhang \cite{Huang2007},  any complex QCQP with two quadratic inequality constraints  has strong duality and has no optimality gap with its SDP relaxation. Besides, Ai, Huang and Zhang \cite{AiZhang2011} presented several new complex rank-one solution theorems. Recently, He, Jiang and Zhu \cite{He2021} extended the rank-one decomposition procedure to the quaternion field.  

This paper aims at the intrinsic difficulty coming from the non-convexity of the following two programs: the real-valued $(HQP_{3})$ and the complex-valued $(HQP_{4})$. We shall establish a uniform sufficient and necessary test condition to detect whether or not the SDP relaxation of the above two programs is tight. All the test conditions given in the papers \cite{AiZhang2009,Cheng2021,Nguyen2019} can be regarded as special cases of the new test condition. Furthermore, by using the new test condition, we shall generalize S-lemma and Yuan's lemma to three real and four complex quadratic forms first under certain exact conditions, which improves Proposition 3.6 of \cite{Telarky2007} and Theorem 3.9 of \cite{Peng1997}. It displays also that the new test condition has potential wide application values.


%

This paper is organized as follows. In the section 2, we review some rank-one decomposition theorems over the real and complex number fields.  In the section 3, the sufficient and necessary test condition is derived. In the section 4, S-lemma and Yuan's lemma are generalized to three real and four complex quadratic forms under certain exact conditions. In the section 5, numerical experiments are presented to show the numerical effectiveness of the test condition.

\noindent\textbf{Notation.} Let $\mathcal{F}$ denote the real number field $\mathcal{R}$ or the complex number field $\mathcal{C}$; And correspondingly,
$\mathcal{F}^n$ denotes the $n$-dimensional real vector space $\mathcal{R}^n$ or complex vector space $\mathcal{C}^n$. For any vector $x\in\mathcal{F}^n$, its transpose and conjugate transpose are described by $x^T$ and $x^H$, respectively. Moreover, $\mathcal{HF}^n$ (or $\mathcal{HF}_+^{n}$) denotes the set of all the $n\times n$ Hermitian matrices ( or all the $n\times n$ positive semidefinite Hermitian matrices) over the number field $\mathcal{F}$. 
For any matrix $A\in\mathcal{HF}^{n}$, the notations $\rank(A)$, $\range(A)$ and $\Null(A)$ denote the rank,  range subspace and null subspace of  $A$, respectively; and $A\succeq0$ (or $\succ0$) means that the matrix $A$ is positive semidefinite (or positive definite). For any two matrices $A,B\in\mathcal{HF}^{n}$, their inner product is denoted by  $A\bullet B=\tr(B^HA)$, where `$\tr(B^HA)$' denotes the trace of the matrix $B^HA$. Finally, we use $v^*(HQP_m)$ to denote the optimal objective value of $(HQP_{m})$.

\section{A review on matrix rank-one decomposition}
Firstly, let us review some classical matrix rank-one decomposition theorems in real and complex number fields. Here we manage to describe them uniformly for both real and complex fields, and to see them in a new perspective. Put
\begin{equation}\label{mf-define}
m_\mathcal{F}=
\begin{cases}
	1, & \text{if }\mathcal{F}=\mathcal{R},\\
	2, & \text{if }\mathcal{F}=\mathcal{C}.
\end{cases}
\end{equation}
The following lemma combines Corollary 4 of \cite{Sturm2003On}  and Theorem 2.1 of \cite{Huang2007}. 
 
\begin{lemma}[Corollary 4 of \cite{Sturm2003On}  and Theorem 2.1 of \cite{Huang2007}]\label{lemma:rank-one}
Assume $A_1,A_{m_{\mathcal{F}}}\in\mathcal{HF}^n$ and $0\neq X\in\mathcal{HF}_+^n$ with  $\rank(X)=r$. Then, in polynomial-time, one can find a rank-one decomposition for $X$:
	$$
	X=x_1x_1^H+x_2x_2^H+\cdots+x_rx_r^H
	$$
	such that  
\begin{equation}\label{eq:AX=0,mf}
   \left(A_1\bullet rx_kx_k^H,\, A_{m_{\mathcal{F}}}\bullet rx_kx_k^H\right)
	= \left( A_1\bullet X,\,A_{m_{\mathcal{F}}}\bullet X \right) \quad \mbox{ for all }k=1,2,\cdots,r.
\end{equation}
\end{lemma}

Note that, if $\mathcal{F}=\mathcal{R}$, $A_{m_{\mathcal{F}}}$ is just $A_1$, that is  \eqref{eq:AX=0,mf} contains only one equation; only if  $\mathcal{F}=\mathcal{C}$,  \eqref{eq:AX=0,mf} contains two equations indeed.

If one more matrix, say $A_{m_{\mathcal{F}}+1}$, is considered, the above perfect rank-one decomposition result appears no longer, but a rank-one solution may exist in certain conditions. The following result is essentially contributed by Lemma 3.3 of \cite{AiZhang2009} and Theorem 2.2 of \cite{AiZhang2011}, but is described in a more refined form.

\begin{lemma}\label{theorem:rank-one+1}
Assume $A_1,A_{m_{\mathcal{F}}},A_{m_{\mathcal{F}}+1}\in\mathcal{HF}^n$ and $0\neq X\in\mathcal{HF}_+^n$. Let $\mathcal{V}\subseteq\mathcal{F}^{n}$ be a linear subspace  satisfying $\range(X)\subseteq\mathcal{V}$ and $\dim(\mathcal{V})\geq3$.
\begin{description}
	\item[(i)] If $(A_1\bullet X,\,A_{m_{\mathcal{F}}}\bullet X,\,A_{m_{\mathcal{F}}+1}\bullet X)\neq(0,0,0)$, then one can find in polynomial-time a nonzero vector $ x\in\range(X)$ such that 
	\begin{equation}\label{neq:Ax=0,mf+1}
		\left(A_1\bullet xx^H,\,A_{m_{\mathcal{F}}}\bullet xx^H,\,A_{m_{\mathcal{F}}+1}\bullet xx^H \right)= 
		\left(A_1\bullet X,\,A_{m_{\mathcal{F}}}\bullet X,\,A_{m_{\mathcal{F}}+1}\bullet X \right).
	\end{equation}
	
	\item[(ii)] If $(A_1\bullet X,\,A_{m_{\mathcal{F}}}\bullet X,\,A_{m_{\mathcal{F}}+1}\bullet X)=(0,0,0)$, then one can find in polynomial-time a nonzero vector $x\in\mathcal{V}$ such that \eqref{neq:Ax=0,mf+1} holds.
\end{description} 
\end{lemma}

Finally, if another more matrix, say $A_{m_{\mathcal{F}}+2}$, is considered also, a rank-one solution may still exist under some stricter conditions.
The following lemma is essentially contributed by Theorem 2.3 of \cite{AiZhang2011}. 

\begin{lemma}\label{theorem:rank-one+2}
Assume $A_1,A_{m_{\mathcal{F}}},A_{m_{\mathcal{F}}+1},A_{m_{\mathcal{F}}+2}\in\mathcal{HF}^n$ and $0\neq X\in\mathcal{HF}_+^n$. Let $\mathcal{V}\subseteq\mathcal{F}^{n}$ be a linear subspace  that satisfies $\range(X)\subseteq\mathcal{V}$, $\,\dim(\mathcal{V})\geq3$ and  
$$
(A_1\bullet Y,A_{m_{\mathcal{F}}}\bullet Y,A_{m_{\mathcal{F}}+1}\bullet Y,A_{m_{\mathcal{F}}+2}\bullet Y)\neq(0,0,0,0),\quad
\forall\, 0\ne Y\in\mathcal{HF}_+^n \mbox{ and } \range(Y)\subseteq\mathcal{V}.
$$
Then one can find in polynomial-time a nonzero vector $x\in\mathcal{V}$ such that  
\begin{equation}
(A_1\bullet xx^H,A_{m_{\mathcal{F}}}\bullet xx^H, A_{m_{\mathcal{F}}+1}\bullet xx^H, A_{m_{\mathcal{F}}+2}\bullet xx^H)
	=(A_1\bullet X,A_{m_{\mathcal{F}}}\bullet X, A_{m_{\mathcal{F}}+1}\bullet X, A_{m_{\mathcal{F}}+2}\bullet X).	
\end{equation}
\end{lemma}

\section{On the tightness of an SDP relaxation}
In this section, we consider the following typical Semi-Definite Programming (SDP) relaxation of $(HQP_m)$:
\begin{equation}
	\begin{array}{lll}
		(SP_m)&\mbox{minimize}   &A_0\bullet X\\
		&\mbox{subject to} &
		\begin{cases}
			A_1\bullet X\unlhd_1 c_1,\\
			A_{2}\bullet X\unlhd_{2} c_{2},\\
			\quad\quad\vdots\\
			A_{m}\bullet X\unlhd_{m} c_{m}\,\,(c_{m}\neq0),\\
			X\succeq0.
		\end{cases}
	\end{array}
\end{equation}
The dual problem of $(SP_m)$ can be written as follows:
\begin{equation}
	\begin{array}{lll}
		(SD_m)&\mbox{maximize}   &-\sum\limits_{i=1}^{m}c_i\mu_i\\
		&\mbox{subject to} & Z=A_0+\sum\limits_{i=1}^{m}\mu_iA_i\succeq0,\\
		&& \mu_i\in\mathcal{R} \text{ and } \mu_i\unrhd_i^*0,\,\,\,i=1,2,\cdots,m,\\
	\end{array}
\end{equation}
where 
$$\unrhd_i^*=
\begin{cases}
	\geq, & \text{if }\unlhd_i \text{ is } \leq;\\
	unrestricted, & \text{if }\unlhd_i \text{ is } =.
\end{cases}
$$

\begin{assumption}\label{ass1}
  	{\bf (i)} $(SP_m)$ satisfies the Slater condition.
  	{ \bf (ii)} $(SD_m)$ satisfies the Slater condition, that is, there exist $m$ real numbers  $ \tilde\mu_1,\tilde\mu_2,\cdots,\tilde\mu_{m}$ such that 
  	\begin{equation}\label{(SD0)-interior}
  		\tilde{Z}=A_0+\sum\limits_{i=1}^{m}\tilde\mu_iA_i\succ0,\,\quad \tilde\mu_i>0\,\,\forall  ``\unlhd_i"= ``\leq"\,(i=1,2,\cdots,m).
  	\end{equation}
\end{assumption}

According to the SDP theory, under Assumption \ref{ass1}, both $(SP_m)$ and $(SD_m)$ must have optimal solutions. A primal feasible solution $\hat{X}$ and a dual feasible solution $(\hat{Z},\hat{\mu}_1,\cdots,\hat{\mu}_m)$ are an optimal solution pair to $(SP_m)$ and $(SD_m)$, if and only if the pair  satisfies the following complementary condition:
\begin{equation}\label{complement-condition}
	\hat{Z}\bullet\hat{X}=0,\quad \hat\mu_i(A_i\bullet \hat{X}-c_i)=0\,\, (i=1,2,\cdots,m).
\end{equation}



The following lemma tells us that, under Assumption \ref{ass1}, if only $(HQP_m)$ has feasible solutions then  $(HQP_m)$ must have optimal solutions.

\begin{lemma}\label{QP-has-optimal-solution}
Let Assumption \ref{ass1} hold,  and let the feasible region $\Omega$ of $(HQP_m)$ be nonempty, say $x_0\in\Omega$. Then the level set $\Omega(x_0)=\{x\,|\, x^TA_0x\leq x_0^TA_0x_0,\,\,x\in\Omega \}$  must be a bounded closed set.
\end{lemma}

\begin{proof}
It is straightforward that $\Omega(x_0)$ is a closed set. So one needs only to prove that $\Omega(x_0)$ is bounded. 
In fact, Assumption \ref{ass1} (ii) means that \eqref{(SD0)-interior} holds.
Let $\lambda_0 >0$ be the smallest eigenvalue of the matrix $\tilde{Z}$ defined in \eqref{(SD0)-interior}.
Therefore, for any $x\in\Omega(x_0)$, we have
\begin{equation*}
  \lambda_0\|x\|^2\leq x^H\tilde{Z}x=x^HA_0x+\sum\limits_{i=1}^{m}\tilde\mu_ix^HA_ix\leq x_0^HA_0x_0+\sum\limits_{i=1}^{m}\tilde\mu_ic_i.
\end{equation*}
The proof is completed.
\end{proof}


From now on, we consider two special cases of $(HQP_{m})$:   $m=3$ for $\mathcal{F}=\mathcal{R}$ and $m=4$ for $\mathcal{F}=\mathcal{C}$. In order to analyze the two cases uniformly, we manage to formulate them  by a unified form as follows:
\begin{equation*}
	\begin{array}{lll}
		(HQP_{m_{\mathcal{F}}+2})&\mbox{minimize}   &x^HA_0x\\
		&\mbox{subject to} &
		\begin{cases}
			x^HA_1x\unlhd_1 c_1,\\
			x^HA_{m_{\mathcal{F}}}x\unlhd_{m_{\mathcal{F}}} c_{m_{\mathcal{F}}}\,\,(\text{only for } \mathcal{F}=\mathcal{C}),\\
			x^HA_{m_{\mathcal{F}}+1}x\unlhd_{m_{\mathcal{F}}+1} c_{m_{\mathcal{F}}+1},\\
			x^HA_{m_{\mathcal{F}}+2}x\unlhd_{m_{\mathcal{F}}+2} c_{m_{\mathcal{F}}+2}\,(c_{m_{\mathcal{F}}+2}\neq0),
		\end{cases}
	\end{array}
\end{equation*}
where the number $m_{\mathcal{F}}$ is defined by \eqref{mf-define}, i.e. $m_{\mathcal{F}}=1$ for $\mathcal{F}=\mathcal{R}$ and $m_{\mathcal{F}}=2$ for $\mathcal{F}=\mathcal{C}$.   
When $\mathcal{F}=\mathcal{R}$, the ``so-called" second constraint ``$x^HA_{m_{\mathcal{F}}}x\unlhd_{m_{\mathcal{F}}} c_{m_{\mathcal{F}}}$" is just the first constraint ``$x^HA_1x\unlhd_1 c_1$", that is, at this moment  $(HQP_{m_{\mathcal{F}}+2})$ has only three constraints. 



The following lemma plays an important role in the remaining discussions,
which shows that, under Assumption \ref{ass1}, the matrices $A_1,A_{m_{\mathcal{F}}},A_{m_{\mathcal{F}}+1},A_{m_{\mathcal{F}}+2}$ are  jointly definite in the null subspace of $\hat{Z}$.

\begin{lemma}\label{lemma:local-joint-definite}
Let Assumption \ref{ass1} hold and $(\hat{Z},\hat{\mu}_1,\hat{\mu}_{m_{\mathcal{F}}},\hat{\mu}_{m_{\mathcal{F}}+1},\hat{\mu}_{m_{\mathcal{F}}+2})$ be an optimal solution to $(SD_{m_{\mathcal{F}}+2})$. Then 
$$
(A_1\bullet Y, A_{m_{\mathcal{F}}}\bullet Y, A_{m_{\mathcal{F}}+1}\bullet Y, A_{m_{\mathcal{F}}+2}\bullet Y)\neq(0,0,0,0),\quad \forall\, 0\neq Y\succeq0 \text{ and } \hat{Z}\bullet Y=0.
$$
\end{lemma}

\begin{proof}
Suppose by contradiction that
\begin{equation*}\label{A1m=0}
	A_i\bullet Y=0,\,\,\,i=1,m_{\mathcal{F}},m_{\mathcal{F}}+1,m_{\mathcal{F}}+2.
\end{equation*}
It implies that
$$
A_0\bullet Y=\hat{Z}\bullet Y- \sum\limits_{i=1}^{m_{\mathcal{F}}+2}\hat\mu_iA_i \bullet Y =0-0=0.
$$ 
On the other hand, Assumption \ref{ass1}(ii) means that \eqref{(SD0)-interior} holds. So one obtains that 
$$
\tilde{Z}\bullet Y= A_0\bullet Y+\sum\limits_{i=1}^{m_{\mathcal{F}}+2}\tilde\mu_iA_i\bullet Y=0+0=0,
$$ 
which contradicts with 
$\tilde{Z}\succ0 \text{ and } 0\neq Y\succeq0$.
The proof is completed.
\end{proof}


The following lemma shows that, if any optimal multiplier that  corresponds to an inequality constraint of $(SP_{m_{\mathcal{F}}+2})$ is equal to zero,  then $(SP_{m_{\mathcal{F}}+2})$ must be tight.

\begin{lemma}\label{lemma:Mu=zero}
Let Assumption \ref{ass1} hold, and let $\hat{X}$ and $(\hat{Z},\hat{\mu}_1,\hat{\mu}_{m_{\mathcal{F}}},\hat{\mu}_{m_{\mathcal{F}}+1},\hat{\mu}_{m_{\mathcal{F}}+2})$ be an optimal solution pair to  $(SP_{m_{\mathcal{F}}+2})$ and $(SD_{m_{\mathcal{F}}+2})$.
If there is one multiplier  $\hat{\mu}_{i_0}$ such that $\hat{\mu}_{i_0}=0$ and $``\unlhd_{i_0}"=``\leq"$, then one can find in polynomial-time an optimal solution with a rank less than $2$ to $(SP_{m_{\mathcal{F}}+2})$. 
\end{lemma}

\begin{proof}
Without loss of generality, we assume $i_0={m_{\mathcal{F}}+2}$, which means
\begin{equation}\label{eq0}
\hat{\mu}_{m_{\mathcal{F}}+2}=0,\,``\unlhd_{m_{\mathcal{F}}+2}"=``\leq", \text{ and }
\tilde{\mu}_{m_{\mathcal{F}}+2}>0, 
\end{equation} 
where $\tilde{\mu}_{m_{\mathcal{F}}+2}$ is defined by \eqref{(SD0)-interior}. 
To complete the proof, one needs only to find an $n$-dimensional vector $x$ satisfying
\begin{equation}\label{eq:MuZeroOptim}
x\in\Null(\hat{Z});\, A_i\bullet xx^H=A_i\bullet\hat{X}\,\text{ for }\,i=1,m_{\mathcal{F}},m_{\mathcal{F}}+1; \text{ and } A_{m_{\mathcal{F}}+2}\bullet xx^H\leq A_{m_{\mathcal{F}}+2}\bullet \hat{X};
\end{equation} 
then $xx^H$ must be an optimal solution to $(SP_{m_{\mathcal{F}}+2})$. So our proof proceeds to two cases as follows.

\noindent\textbf{Case 1.} {\bf $(A_1\bullet \hat{X},A_{m_{\mathcal{F}}}\bullet \hat{X}, A_{m_{\mathcal{F}}+1}\bullet \hat{X})=(0,0,0).$}\\
Notice that 
$$
A_0\bullet \hat{X}=\hat{Z}\bullet \hat{X}-\sum\limits_{i=1}^{m_{\mathcal{F}}+2}\hat\mu_iA_i \bullet \hat{X}=0-0=0.
$$
From \eqref{(SD0)-interior}, one can obtain that
$$
0\leq \tilde{Z}\bullet \hat{X}= A_0\bullet \hat{X}+ \sum\limits_{i=1}^{m_{\mathcal{F}}+2}\tilde\mu_iA_i\bullet \hat{X}= \tilde\mu_{m_{\mathcal{F}}+2} A_{m_{\mathcal{F}}+2}\bullet\hat{X},
$$
which, together with $\tilde\mu_{m_{\mathcal{F}}+2}>0$ from \eqref{eq0}, deduces that
$$
A_{m_{\mathcal{F}}+2}\bullet\hat{X}\geq0. 
$$
Then it is easily verified that the zero vector $x=0$ satisfies \eqref{eq:MuZeroOptim}.

\noindent\textbf{Case 2.} {\bf $(A_1\bullet \hat{X},A_{m_{\mathcal{F}}}\bullet \hat{X}, A_{m_{\mathcal{F}}+1}\bullet \hat{X})\neq (0,0,0).$}\\ 
For this case, there must be $ r:=\rank(\hat{X})\geq1$. Denote $$\delta_i:=A_i\bullet\hat{X}, \quad i=1,m_{\mathcal{F}},m_{\mathcal{F}}+1,m_{\mathcal{F}}+2;$$
and without loss of generality, we assume that $\delta_{m_{\mathcal{F}}+1}=A_{m_{\mathcal{F}}+1}\bullet \hat{X}\neq0$. Then we have
\begin{equation*}
\left(A_i-\dfrac{\delta_i}{\delta_{m_{\mathcal{F}}+1}}A_{m_{\mathcal{F}}+1}\right)\bullet\hat{X}=0,
\quad\,i=1,m_{\mathcal{F}},m_{\mathcal{F}}+2.
\end{equation*}
By Lemma \ref{lemma:rank-one}, there is a rank-one decomposition of $\hat{X}$,
$
\hat{X}=\hat x_1\hat x_1^H+\hat x_2\hat x_2^H+\cdots+\hat x_r\hat x_r^H,
$
such that
\begin{equation}\label{eq1}
\left\{
\begin{array}{ll}
\left(A_1-\dfrac{\delta_1}{\delta_{m_{\mathcal{F}}+1}}A_{m_{\mathcal{F}}+1}\right)\bullet  \hat x_k\hat x_k^H=\left(A_{m_{\mathcal{F}}}-\dfrac{\delta_{m_{\mathcal{F}}}}{\delta_{m_{\mathcal{F}}+1}}A_{m_{\mathcal{F}}+1}\right)\bullet \hat x_k\hat x_k^H=0,\,\,\forall\,k=1,2,\cdots,r;\\
\left(A_{m_{\mathcal{F}}+2}-\dfrac{\delta_{m_{\mathcal{F}}+2}}{\delta_{m_{\mathcal{F}}+1}}A_{m_{\mathcal{F}}+1}\right)\bullet \hat x_{k_0}\hat x_{k_0}^H\leq0,\qquad \text{ for some index }
k_0\in \{1,2,\cdots,r\}.
\end{array}
\right.
\end{equation}
Put $t_0:=A_{m_{\mathcal{F}}+1}\bullet \hat x_{k_0}\hat x_{k_0}^H/{\delta_{m_{\mathcal{F}}+1}}$. Then, from \eqref{eq1}, one has
\begin{equation}\label{eq2}
	A_i\bullet \hat x_{k_0}\hat x_{k_0}^H=t_0\delta_i,\,i=1,m_{\mathcal{F}},m_{\mathcal{F}}+1; \text{ and } 	A_{m_{\mathcal{F}}+2}\bullet \hat x_{k_0}\hat x_{k_0}^H\leq t_0\delta_{m_{\mathcal{F}}+2}.
\end{equation}
We assert that $t_0>0$ because,  if $t_0\leq0$,  then $0\neq \hat x_{k_0}\hat x_{k_0}^H-t_0\hat{X}\succeq0$ and one can   from \eqref{eq2} obtain that 
\begin{equation*}
A_i\bullet\left(\hat x_{k_0}\hat x_{k_0}^H-t_0\hat{X}\right)=0,\,\,\,i=1,m_{\mathcal{F}},m_{\mathcal{F}}+1; \text{ and } A_{m_{\mathcal{F}}+2}\bullet\left(\hat x_{k_0}\hat x_{k_0}^H-t_0\hat{X}\right)\leq0\,;
\end{equation*}
which, together with \eqref{(SD0)-interior} and \eqref{eq0}, leads to a contradiction as follows:
$$
\begin{array}{lll}
0&<&\tilde{Z}\bullet\left(\hat x_{k_0}\hat x_{k_0}^H-t_0\hat{X}\right)=A_0\bullet\left(\hat x_{k_0}\hat x_{k_0}^H-t_0\hat{X}\right)+\sum\limits_{i=1}^{m_{\mathcal{F}}+2}\tilde\mu_iA_i\bullet\left(\hat x_{k_0}\hat x_{k_0}^H-t_0\hat{X}\right)\\
&\leq& A_0\bullet\left(\hat x_{k_0}\hat x_{k_0}^H-t_0\hat{X}\right)
=\hat{Z}\bullet\left(\hat x_{k_0}\hat x_{k_0}^H-t_0\hat{X}\right)-\sum\limits_{i=1}^{m_{\mathcal{F}}+2}\hat\mu_iA_i\bullet\left(\hat x_{k_0}\hat x_{k_0}^H-t_0\hat{X}\right)\\
&=&0-0=0.
\end{array}
$$
Hence $t_0>0$. Therefore, from \eqref{eq2}, the rank-one matrix $\hat x_{k_0}\hat x_{k_0}^H/t_0$ satisfies \eqref{eq:MuZeroOptim}.
The proof is completed.
\end{proof}

The following definition plays a pivotal role in our main result.

\begin{definition}\label{I-definition}
Let $\hat{X}$ and $(\hat{Z},\hat{\mu}_1,\hat{\mu}_{m_{\mathcal{F}}},\hat{\mu}_{m_{\mathcal{F}}+1},\hat{\mu}_{m_{\mathcal{F}}+2})$ be an optimal solution pair to $(SP_{m_{\mathcal{F}}+2})$ and $(SD_{m_{\mathcal{F}}+2})$. We say that this pair has Property $\mathcal{I}$ if the following conditions are simultaneously satisfied:
\begin{description}
	\item[(I.1)] $\hat{\mu}_i\neq0$ for each   $``\unlhd_i"=``\leq"$\,\,$\left(i=1,m_{\mathcal{F}}, m_{\mathcal{F}}+1,m_{\mathcal{F}}+2\right)$, that is, all the optimal multipliers corresponding to the inequality constraints are nonzero;
	
	\item[(I.2)]  $\rank(\hat{Z})=n-2$;
	
	\item[(I.3)]  $\rank(\hat{X})=2$;
	
	\item[(I.4)]  there is a rank-one decomposition of $\hat{X}$, $\hat{X}=\hat{x}_1\hat{x}_1^H+\hat{x}_2\hat{x}_2^H$, such that 
\begin{description}
	\item[(I.4.1)]
	 $\left(A_1-\dfrac{c_1}{c_{m_{\mathcal{F}}+2}}
	 A_{m_{\mathcal{F}}+2}\right)\bullet \hat{x}_1\hat{x}_1^H=
	 \left(A_1-\dfrac{c_1}{c_{m_{\mathcal{F}}+2}}A_{m_{\mathcal{F}}+2}\right)\bullet \hat{x}_2\hat{x}_2^H=0$\\ 
	 \quad\text{ and }
	 $\re\left(\hat{x}_1^H\left(A_1-\dfrac{c_1}{c_{m_{\mathcal{F}}+2}}A_{m_{\mathcal{F}}+2}\right)\hat{x}_2\right)\neq0;$
	 
	 \item[(I.4.2)] {\bf (only for $\mathcal{F}=\mathcal{C}$)} 
	 
	  $\quad\,\,\left(A_{m_{\mathcal{F}}}-\dfrac{c_{m_{\mathcal{F}}}}{c_{m_{\mathcal{F}}+2}}
	 A_{m_{\mathcal{F}}+2}\right)\bullet \hat{x}_1\hat{x}_1^H=
	 \left(A_{m_{\mathcal{F}}}-\dfrac{c_{m_{\mathcal{F}}}}{c_{m_{\mathcal{F}}+2}}A_{m_{\mathcal{F}}+2}\right)\bullet \hat{x}_2\hat{x}_2^H=0,$
	 
	 $\quad\,\,\re\left(\hat{x}_1^H\left(A_{m_{\mathcal{F}}}-\dfrac{c_{m_{\mathcal{F}}}}{c_{m_{\mathcal{F}}+2}}A_{m_{\mathcal{F}}+2}\right)\hat{x}_2\right)=0,$
	 
	  $\quad\,\,\im\left(\hat{x}_1^H\left(A_{m_{\mathcal{F}}}-\dfrac{c_{m_{\mathcal{F}}}}{c_{m_{\mathcal{F}}+2}}A_{m_{\mathcal{F}}+2}\right)\hat{x}_2\right)\neq0;$
	  
	  \item[(I.4.3)] 
	  $\left(\left(A_{m_{\mathcal{F}}+1}-\dfrac{c_{m_{\mathcal{F}}+1}}{c_{m_{\mathcal{F}}+2}}
	  A_{m_{\mathcal{F}}+2}\right)\bullet \hat{x}_1\hat{x}_1^H\right)
	  \left(\left(A_{m_{\mathcal{F}}+1}-\dfrac{c_{m_{\mathcal{F}}+1}}{c_{m_{\mathcal{F}}+2}}A_{m_{\mathcal{F}}+2}\right)\bullet \hat{x}_2\hat{x}_2^H\right)<0.$
\end{description}
\end{description}
\end{definition}

\begin{remark}
	When $\mathcal{F}=\mathcal{R}$, {\bf (I.4)} of Property $\mathcal{I}$ can be reformulated in a concise form:
	\begin{description}
%
%
		
		\item[(I.4)]  there is a rank-one decomposition of $\hat{X}$, $\hat{X}=\hat{x}_1\hat{x}_1^T+\hat{x}_2\hat{x}_2^T$, such that 
		\begin{description}
			\item[(I.4.1)]
			$\left(A_1-\dfrac{c_1}{c_3}A_3\right)
			\bullet \hat{x}_1\hat{x}_1^T=
			\left(A_1-\dfrac{c_1}{c_3}A_3\right)\bullet \hat{x}_2\hat{x}_2^T=0\,\text{ and }\,\hat{x}_1^T\left(A_1-\dfrac{c_1}{c_3}A_3\right)\hat{x}_2\neq0;$
			
			\item[(I.4.2)] 
			$\left(\left(A_2-\dfrac{c_2}{c_3}A_3\right)
			\bullet \hat{x}_1\hat{x}_1^T\right)
			\left(\left(A_2-\dfrac{c_2}{c_3}A_3\right)
			\bullet \hat{x}_2\hat{x}_2^T\right)<0.$
		\end{description}
	\end{description}
\end{remark}
Certainly,  Property $\mathcal{I}$ in \cite{AiZhang2009}, Property $\mathcal{J}$ in \cite{Nguyen2019} and 
Property $\mathcal{I}^+$ in \cite{Cheng2021} are included in the above  Property $\mathcal{I}$ as three special cases. 

Right now we are ready to state the main result of the paper. Here $v^*(SP_{m_{\mathcal{F}}+2})$ and $v^*(HQP_{m_{\mathcal{F}}+2})$  denote the optimal objective values of $(SP_{m_{\mathcal{F}}+2})$ and $(HQP_{m_{\mathcal{F}}+2})$, respectively. And $\Omega$ denotes the feasible set of $(HQP_{m_{\mathcal{F}}+2})$.

\begin{theorem}\label{theorem:main-result}
Let Assumption \ref{ass1} hold, and let $\hat{X}$ and $(\hat{Z},\hat{\mu}_1,\hat{\mu}_{m_{\mathcal{F}}},\hat{\mu}_{m_{\mathcal{F}}+1},\hat{\mu}_{m_{\mathcal{F}}+2})$ be any optimal solution pair  of $(SP_{m_{\mathcal{F}}+2})$ and $(SD_{m_{\mathcal{F}}+2})$. Then, $v^*(SP_{m_{\mathcal{F}}+2})<v^*(HQP_{m_{\mathcal{F}}+2})$ or alternatively $\Omega=\emptyset$, if and only if the pair satisfies Property $\mathcal{I}$. Moreover, if the pair does not satisfy Property $\mathcal{I}$, then one can find a global optimal solution to $(HQP_{m_{\mathcal{F}}+2})$ from the pair in polynomial-time.
\end{theorem}

\begin{proof}
$``\Longleftarrow"$ {\bf\large Sufficiency.}\\
One needs to verify that, if the pair satisfies Property $\mathcal{I}$, then $v^*(SP_{m_{\mathcal{F}}+2})<v^*(HQP_{m_{\mathcal{F}}+2})$ or alternatively $\Omega=\emptyset$. We shall  complete the verification by contradiction. Assume that $\Omega\neq\emptyset$ and  $v^*(SP_{m_{\mathcal{F}}+2})=v^*(HQP_{m_{\mathcal{F}}+2})$.  By Lemma \ref{QP-has-optimal-solution},  $(HQP_{m_{\mathcal{F}}+2})$ has at least one optimal solution, say $x^*$.  Then $x^*x^{*^H}$ is also an optimal solution of $(SP_{m_{\mathcal{F}}+2})$. Due to Property $\mathcal{I}$ (I.1), the primal optimal solution $x^*x^{*^H}$ and the dual optimal solution $(\hat{Z},\hat{\mu}_1,\hat{\mu}_{m_{\mathcal{F}}},\hat{\mu}_{m_{\mathcal{F}}+1},\hat{\mu}_{m_{\mathcal{F}}+2})$ must satisfy the complementary condition as follows:
\begin{equation}\label{KKT_condition}
	\hat{Z}\bullet x^*x^{*^H}=0,\ A_{i}\bullet x^*x^{*^H}=c_i,\,i=1,m_{\mathcal{F}},m_{\mathcal{F}}+1,m_{\mathcal{F}}+2.
\end{equation}
Notice that $\hat{Z}\bullet x^*x^{*^H}=0$ means $x^*\in \Null(\hat{Z})$. And both conditions $ ``\rank(\hat{Z})=n-2"$ and $``\rank(\hat{X})=2"$ imply $L(\hat{x}_1,\hat{x}_2)=\Null(\hat{Z})$, 
where $L(\hat{x}_1,\hat{x}_2)$ denotes the linear subspace spanned by the vectors $\hat{x}_1$ and $\hat{x}_2$.
Hence $x^*$ must be a linear combination of $\hat{x}_1$ and $\hat{x}_2$, say
\begin{equation}\label{linear_combination}
x^*=\alpha\hat{x}_1+\beta\hat{x}_2,\quad\alpha,\beta\in\mathcal{F}.
\end{equation}
From \eqref{KKT_condition} and \eqref{linear_combination}, one obtains
\begin{equation}\label{eq13}
\begin{cases}
\left(A_1-\dfrac{c_1}{c_{m_{\mathcal{F}}+2}}
A_{m_{\mathcal{F}}+2}\right)\bullet (\alpha\hat{x}_1+\beta\hat{x}_2)(\alpha\hat{x}_1+\beta\hat{x}_2)^{H}=0,\\
\left(A_{m_{\mathcal{F}}}-\dfrac{c_{m_{\mathcal{F}}}}{c_{m_{\mathcal{F}}+2}}A_{m_{\mathcal{F}}+2}\right)
\bullet (\alpha\hat{x}_1+\beta\hat{x}_2)(\alpha\hat{x}_1+\beta\hat{x}_2)^{H}=0.
\end{cases}
\end{equation}
By using (I.4) of Property $\mathcal{I}$, \eqref{eq13} can be reduced into the following form:
\begin{equation}\label{eq17}
\begin{cases}
	\re\left(\hat{x}_1^H(A_1-\dfrac{c_1}{c_{m_{\mathcal{F}}+2}}
	A_{m_{\mathcal{F}}+2}) \hat{x}_2\,\overline{\alpha}\beta\right)=0\\
	\im\left(\hat{x}_1^H(A_{m_{\mathcal{F}}}-\dfrac{c_{m_{\mathcal{F}}}}{c_{m_{\mathcal{F}}+2}}A_{m_{\mathcal{F}}+2})\hat{x}_2\right)\im\left(\overline{\alpha}\beta\right)=0
\end{cases}
\Longrightarrow\,\,
 \overline\alpha\beta=0,
\end{equation}
which implies either $\alpha=0$ or $\beta=0$. Without loss of generality, we assume $\beta=0$, that is $x^*=\alpha\hat{x}_1$. Substituting $x^*=\alpha\hat{x}_1$ into \eqref{KKT_condition}, one yields
$$
\begin{cases}
	0\neq c_{m_{\mathcal{F}}+2}= A_{m_{\mathcal{F}}+2}\bullet x^*x^{*^H}=|\alpha|^2A_{m_{\mathcal{F}}+2}\bullet \hat x_1\hat x_1^H \Longrightarrow\,\,
	\alpha\neq0,\\
	0=\left(A_{m_{\mathcal{F}}+1}-\dfrac{c_{m_{\mathcal{F}}+1}}{c_{m_{\mathcal{F}}+2}}
	A_{m_{\mathcal{F}}+2}\right)\bullet x^*x^{*^H} 
	=|\alpha|^2\left(A_{m_{\mathcal{F}}+1}-\dfrac{c_{m_{\mathcal{F}}+1}}{c_{m_{\mathcal{F}}+2}}
	A_{m_{\mathcal{F}}+2}\right)\bullet \hat{x}_1\hat{x}_1^H\\
	\Longrightarrow\,\,\left(A_{m_{\mathcal{F}}+1}-\dfrac{c_{m_{\mathcal{F}}+1}}{c_{m_{\mathcal{F}}+2}}
	A_{m_{\mathcal{F}}+2}\right)\bullet \hat{x}_1\hat{x}_1^H=0,
\end{cases}
$$
which contradicts with Property $\mathcal{I}$ (I.4.3).
So there must be $v^*(SP_{m_{\mathcal{F}}+2})<v^*(HQP_{m_{\mathcal{F}}+2})$ when $\Omega\neq\emptyset$.

\noindent$``\Longrightarrow"$ {\bf\large Necessity.}\\
We shall prove the ``Necessity"  by contradiction: if Property $\mathcal{I}$ fails, then $\Omega\neq\emptyset$ and $v^*(SP_{m_{\mathcal{F}}+2})=v^*(HQP_{m_{\mathcal{F}}+2})$, that is one can find an optimal solution to $(SP_{m_{\mathcal{F}}+2})$ with a rank less than $2$. 

Firstly, if $\rank(\hat{X})< 2$, then $\hat{X}$ is just a desired optimal solution to $(SP_{m_{\mathcal{F}}+2})$. So we assume 
\begin{equation}\label{eq:rankXgeq2}
\rank(\hat{X})\geq2,
\end{equation}
which yields
\begin{equation}\label{eq:rankZleqn-2}
\rank(\hat{Z})\leq n-2,
\end{equation}
due to $\range(\hat{X})\subseteq\Null(\hat{Z})$.  Our proof proceeds to the following three cases.

	
\noindent\textbf{Case 1.}  {\bf Property $\mathcal{I}$ {\textbf{(I.1)}} fails.}\\
It means that there is one multiplier  $\hat{\mu}_{i_0}$ such that $\hat{\mu}_{i_0}=0$ and $``\unlhd_{i_0}"=``\leq"$. For this case, Lemma \ref{lemma:Mu=zero} tells us that one can find an optimal solution  with a rank less than $2$ to $(SP_{m_{\mathcal{F}}+2})$.

\noindent\textbf{Case 2.} {\bf Property $\mathcal{I}$ {\rm\textbf{(I.1)}} holds, but at least one of {\rm\textbf{(I.2)}} and {\rm\textbf{(I.3)}} fails.}\\
As Property $\mathcal{I}$ {\rm\text{(I.1)}} holds, the complementary condition of $\hat{X}$ and $(\hat{Z},\hat{\mu}_1,\hat{\mu}_{m_{\mathcal{F}}},\hat{\mu}_{m_{\mathcal{F}}+1},\hat{\mu}_{m_{\mathcal{F}}+2})$ turns into the following form: 
\begin{equation}\label{eq:PI(I.1)+hatXcomplem}
\left\{
	\begin{array}{l}
		\hat{Z}\bullet\hat{X}=0,\\
		A_{i}\bullet\hat{X}=c_i,\,i=1,m_{\mathcal{F}},m_{\mathcal{F}}+1,m_{\mathcal{F}}+2.
	\end{array}
\right.
\end{equation}
Moreover, by \eqref{eq:rankXgeq2} and \eqref{eq:rankZleqn-2}, a violation of Property $\mathcal{I}$ {\rm\text{(I.2)}} or {\rm\text{(I.3)}} must yield $\dim(\Null(\hat{Z}))\geq3$. Note that $c_{m_{\mathcal{F}}+2}\neq0.$  Therefore, by Lemma \ref{lemma:local-joint-definite} and Lemma \ref{theorem:rank-one+2}, one can find a nonzero vector $x\in \Null(\hat{Z})$ such that  
	\begin{equation*}
		A_i\bullet xx^H=c_i,\,i=1,m_{\mathcal{F}},m_{\mathcal{F}}+1,m_{\mathcal{F}}+2.
	\end{equation*}
Thus $xx^H$ is feasible for $(SP_{m_{\mathcal{F}}+2})$ and satisfies the complementary condition \eqref{eq:PI(I.1)+hatXcomplem}, which means that the rank-one matrix $xx^H$ is an optimal solution to $(SP_{m_{\mathcal{F}}+2})$.
	
\noindent\textbf{Case 3.} {\bf  {\rm\textbf{(I.1)}}, {\rm\textbf{(I.2)}} and {\rm\textbf{(I.3)}} of Property $\mathcal{I}$ hold, but {\rm\textbf{(I.4)}} fails.}\\
In this case, the complementary condition is still described by \eqref{eq:PI(I.1)+hatXcomplem}, and both conditions $``\rank(\hat{Z})=n-2"$ and $``\rank(\hat{X})=2"$ hold. Note that 
\begin{equation*}
	\left(A_i-\dfrac{c_i}{c_{m_{\mathcal{F}}+2}}A_{m_{\mathcal{F}}+2}\right)\bullet\hat{X}=0,\,i=1,m_{\mathcal{F}}, m_{\mathcal{F}}+1.
\end{equation*}
By Lemma \ref{lemma:rank-one}, there is a rank-one decomposition of $\hat{X}$, $\hat{X}=\hat{x}_1\hat{x}_1^H+\hat{x}_2\hat{x}_2^H$, such that
\begin{subequations}
   \begin{align}
		\left(A_1-\dfrac{c_1}{c_{m_{\mathcal{F}}+2}}
		A_{m_{\mathcal{F}}+2}\right)\bullet \hat{x}_1\hat{x}_1^H=
		\left(A_1-\dfrac{c_1}{c_{m_{\mathcal{F}}+2}}A_{m_{\mathcal{F}}+2}\right)\bullet \hat{x}_2\hat{x}_2^H=0,\label{eq4}\\
		\left(A_{m_{\mathcal{F}}}-\dfrac{c_{m_{\mathcal{F}}}}{c_{m_{\mathcal{F}}+2}}
		A_{m_{\mathcal{F}}+2}\right)\bullet \hat{x}_1\hat{x}_1^H=
		\left(A_{m_{\mathcal{F}}}-\dfrac{c_{m_{\mathcal{F}}}}{c_{m_{\mathcal{F}}+2}}A_{m_{\mathcal{F}}+2}\right)\bullet \hat{x}_2\hat{x}_2^H=0,\label{eq5}\\
		\re\left(\hat{x}_1^H\left(A_{m_{\mathcal{F}}}-\dfrac{c_{m_{\mathcal{F}}}}{c_{m_{\mathcal{F}}+2}}
		A_{m_{\mathcal{F}}+2}\right)\hat{x}_2\right)=0. \label{eq6}
	\end{align}
\end{subequations}
Here \eqref{eq6} holds only for $\mathcal{F}=\mathcal{C}$, which can be obtained by resetting $\hat{x}_2:=e^{i\,\theta_0}\hat{x}_2$ and by choosing an appropriate phase value $\theta_0$. Due to \eqref{eq4}, \eqref{eq5} and \eqref{eq6}, the statement ``Property $\mathcal{I}$ (I.4) fails" means that  at least one of the following three sub-statements is true:\vspace{1mm}\\
 ``(I.4.3) fails";  $``\re\left(\hat{x}_1^H(A_1-\dfrac{c_1}{c_{m_{\mathcal{F}}+2}}A_{m_{\mathcal{F}}+2})\hat{x}_2\right)=0"$; $``\im\left(\hat{x}_1^H(A_{m_{\mathcal{F}}}-\dfrac{c_{m_{\mathcal{F}}}}{c_{m_{\mathcal{F}}+2}}A_{m_{\mathcal{F}}+2})\hat{x}_2\right)=0"$.\vspace{1mm}\\  
 So our proof proceeds to the following three subcases. 

\noindent\textbf{Subcase 3.1. Property  $\mathcal{I}$ (I.4.3) fails.} \\ 
It means that $$\left((A_{m_{\mathcal{F}}+1}-\dfrac{c_{m_{\mathcal{F}}+1}}{c_{m_{\mathcal{F}}+2}}
A_{m_{\mathcal{F}}+2})\bullet \hat{x}_1\hat{x}_1^H\right)
\left((A_{m_{\mathcal{F}}+1}-\dfrac{c_{m_{\mathcal{F}}+1}}{c_{m_{\mathcal{F}}+2}}A_{m_{\mathcal{F}}+2})\bullet \hat{x}_2\hat{x}_2^H\right)\geq0.$$ 
Note that $$\left(A_{m_{\mathcal{F}}+1}-\dfrac{c_{m_{\mathcal{F}}+1}}{c_{m_{\mathcal{F}}+2}}
A_{m_{\mathcal{F}}+2}\right)\bullet \hat{x}_1\hat{x}_1^H
+\left(A_{m_{\mathcal{F}}+1}-\dfrac{c_{m_{\mathcal{F}}+1}}{c_{m_{\mathcal{F}}+2}}A_{m_{\mathcal{F}}+2}\right)\bullet \hat{x}_2\hat{x}_2^H
=\left(A_{m_{\mathcal{F}}+1}-\dfrac{c_{m_{\mathcal{F}}+1}}{c_{m_{\mathcal{F}}+2}}
A_{m_{\mathcal{F}}+2}\right)\bullet\hat{X}=0,$$
which deduces that $$\left(A_{m_{\mathcal{F}}+1}-\dfrac{c_{m_{\mathcal{F}}+1}}{c_{m_{\mathcal{F}}+2}}
A_{m_{\mathcal{F}}+2}\right)\bullet \hat{x}_1\hat{x}_1^H=\left(A_{m_{\mathcal{F}}+1}-\dfrac{c_{m_{\mathcal{F}}+1}}{c_{m_{\mathcal{F}}+2}}A_{m_{\mathcal{F}}+2}\right)\bullet \hat{x}_2\hat{x}_2^H=0.$$
Thus, from \eqref{eq4} and \eqref{eq5}, one has 
\begin{equation}\label{eq7}
	\left(A_i-\dfrac{c_i}{c_{m_{\mathcal{F}}+2}}
	A_{m_{\mathcal{F}}+2}\right)\bullet \hat{x}_1\hat{x}_1^H=
	\left(A_i-\dfrac{c_i}{c_{m_{\mathcal{F}}+2}}A_{m_{\mathcal{F}}+2}\right)\bullet \hat{x}_2\hat{x}_2^H=0,\,\,i=1,m_{\mathcal{F}},m_{\mathcal{F}}+1.
\end{equation}
Put $t_k:={A_{m_{\mathcal{F}}+2}\bullet \hat{x}_k\hat{x}_k^H}/{c_{m_{\mathcal{F}}+2}}$ ($k=1,2$). Then \eqref{eq7} becomes 
\begin{equation}\label{eq8}
	\hat{Z}\bullet \hat{x}_k\hat{x}_k^H=0,\,k=1,2;\,\,\,A_i\bullet \hat{x}_k\hat{x}_k^H=t_kc_i,\,\,k=1,2 \text{ and } i=1,m_{\mathcal{F}},m_{\mathcal{F}}+1,m_{\mathcal{F}}+2.
\end{equation}
One can assert that  $t_k>0$ for each $k$ because, if $t_k\leq0$, then $0\neq \hat{x}_k\hat{x}_k^H-t_k\hat{X}\succeq0$ and one obtains that
$$
\hat{Z}\bullet (\hat{x}_k\hat{x}_k^H-t_k\hat{X})=0,\,A_i\bullet (\hat{x}_k\hat{x}_k^H-t_k\hat{X})=0\,\,(i=1,m_{\mathcal{F}},m_{\mathcal{F}}+1,m_{\mathcal{F}}+2),
$$
which contradicts with Lemma \ref{lemma:local-joint-definite}.
Hence  $t_k>0$ ($k=1,2$). 
Therefore, \eqref{eq8} implies that both  $\hat{x}_1\hat{x}_1^H/t_1$ and $\hat{x}_2\hat{x}_2^H/t_2$  are two rank-one optimal solutions to $(SP_{m_{\mathcal{F}}+2})$.

\noindent\textbf{Subcase 3.2. (I.4.3) holds, but $\re\left(\hat{x}_1^H\left(A_1-\dfrac{c_1}{c_{m_{\mathcal{F}}+2}}A_{m_{\mathcal{F}}+2}\right)\hat{x}_2\right)=0$ for  (I.4.1).}\\
This subcase, together with \eqref{eq4},\eqref{eq5} and \eqref{eq6}, makes the following formula hold:
\begin{equation}\label{eq16}
	(\alpha\hat{x}_1+\beta\hat{x}_2)^H(A_i-\dfrac{c_i}{c_{m_{\mathcal{F}}+2}}A_{m_{\mathcal{F}}+2})(\alpha\hat{x}_1+\beta\hat{x}_2)=0,\,i=1,m_\mathcal{F};\,\forall\, \alpha,\beta\in \mathcal{R}.
\end{equation}
 Put 
\begin{equation}\label{eq16-1}
\alpha=\dfrac{t}{\sqrt{1+t^2}} \text{ and } \beta=\dfrac{1}{\sqrt{1+t^2}},
\end{equation}
where $t$ is an unknown real number.
Then, substituting \eqref{eq16-1} into the following equation:
\begin{equation*}
	(\alpha\hat{x}_1+\beta\hat{x}_2)^H(A_{m_{\mathcal{F}}+1}-\dfrac{c_{m_{\mathcal{F}}+1}}{c_{m_{\mathcal{F}}+2}}A_{m_{\mathcal{F}}+2})(\alpha\hat{x}_1+\beta\hat{x}_2)	=0,
\end{equation*}
one obtains a quadratic equation of $t$: 
\begin{equation}\label{eq:t-equation}
at^2+2bt+c=0,
\end{equation}
where
$$
\begin{array}{ll}
a=\hat{x}_1^H\left(A_{m_{\mathcal{F}}+1}-\dfrac{c_{m_{\mathcal{F}}+1}}{c_{m_{\mathcal{F}}+2}}A_{m_{\mathcal{F}}+2}\right)\hat{x}_1,\\
b=\re\left(\hat{x}_1^H\left(A_{m_{\mathcal{F}}+1}-\dfrac{c_{m_{\mathcal{F}}+1}}{c_{m_{\mathcal{F}}+2}}A_{m_{\mathcal{F}}+2}\right)\hat{x}_2\right),\\
c=\hat{x}_2^H\left(A_{m_{\mathcal{F}}+1}-\dfrac{c_{m_{\mathcal{F}}+1}}{c_{m_{\mathcal{F}}+2}}A_{m_{\mathcal{F}}+2}\right)\hat{x}_2.
\end{array}
$$
As Property $\mathcal{I}$ (I.4.3) holds, the equation \eqref{eq:t-equation} must have two distinct real roots with opposite signs. Let $\hat{t}$ be the positive real root. Define 
$$\bar{x}_1:=\dfrac{\hat{t}\hat{x}_1+\hat{x}_2}{\sqrt{1+\hat{t}^2}},\quad \bar{x}_2:=\dfrac{-\hat{x}_1+\hat{t}\hat{x}_2}{\sqrt{1+\hat{t}^2}}.$$
Then we find a new rank-one decomposition of $\hat{X}$, $\hat{X}=\bar{x}_1\bar{x}_1^H+\bar{x}_2\bar{x}_2^H$, such that
\begin{equation}\label{eq11}
	\left(A_i-\dfrac{c_i}{c_{m_{\mathcal{F}}+2}}
	A_{m_{\mathcal{F}}+2}\right)\bullet \bar{x}_1\bar{x}_1^H=
	\left(A_i-\dfrac{c_i}{c_{m_{\mathcal{F}}+2}}A_{m_{\mathcal{F}}+2}\right)\bullet \bar{x}_2\bar{x}_2^H=0,\,\,i=1,m_{\mathcal{F}},m_{\mathcal{F}}+1.
\end{equation}
Note that \eqref{eq11} is just as same as \eqref{eq7}. Thus, following the proof process of Subcase 3.1, one can also find a rank-one optimal solution of $(SP_{m_{\mathcal{F}}+2})$.

\noindent\textbf{Subcase 3.3.\,  (I.4.1) and (I.4.3) hold, but $\im\left(\hat{x}_1^H\left(A_{m_{\mathcal{F}}}-\dfrac{c_{m_{\mathcal{F}}}}{c_{m_{\mathcal{F}}+2}}
	A_{m_{\mathcal{F}}+2}\right)\hat{x}_2\right)=0$ for (I.4.2)\,\,\,  (only for $\mathcal{F}=\mathcal{C}$).}\\
This subcase is considered only for $\mathcal{F}=\mathcal{C}$.    $``\im\left(\hat{x}_1^H\left(A_{m_{\mathcal{F}}}-\dfrac{c_{m_{\mathcal{F}}}}{c_{m_{\mathcal{F}}+2}}
A_{m_{\mathcal{F}}+2}\right)\hat{x}_2\right)=0"$, combined with \eqref{eq5} and \eqref{eq6}, implies that  
\begin{equation}\label{eq18-0}
	\left(A_{m_{\mathcal{F}}}-\dfrac{c_{m_{\mathcal{F}}}}{c_{m_{\mathcal{F}}+2}}
	A_{m_{\mathcal{F}}+2}\right)\bullet\, (\alpha\hat{x}_1+\beta\hat{x}_2) (\alpha\hat{x}_1+\beta\hat{x}_2)^H=0, \,\forall\, \alpha,\beta\in \mathcal{C}.
\end{equation}
Note that $L(\hat{x}_1,\hat{x}_2)=\range(\hat{X})$, where $L(\hat{x}_1,\hat{x}_2)$ denotes the linear subspace spanned by the vectors $\hat{x}_1$ and $\hat{x}_2$. Thus 
\eqref{eq18-0} is exactly equivalent to  
\begin{equation}\label{eq18}
	A_{m_{\mathcal{F}}}\bullet xx^H-\dfrac{c_{m_{\mathcal{F}}}}{c_{m_{\mathcal{F}}+2}}
	A_{m_{\mathcal{F}}+2}\bullet xx^H=\left(A_{m_{\mathcal{F}}}-\dfrac{c_{m_{\mathcal{F}}}}{c_{m_{\mathcal{F}}+2}}
	A_{m_{\mathcal{F}}+2}\right)\bullet xx^H =0, \,\forall\, x\in\range(\hat{X}).  
\end{equation}
Applying Lemma \ref{theorem:rank-one+1} (i) to the following system:
$$
A_1\bullet \hat{X}=c_1,\,\,\,A_{m_{\mathcal{F}}+1}\bullet \hat{X}=c_{m_{\mathcal{F}}+1},\,\,\,A_{m_{\mathcal{F}}+2}\bullet \hat{X}=c_{m_{\mathcal{F}}+2}\,\,(c_{m_{\mathcal{F}}+2}\neq0),
$$
one can find a vector $x\in\range(\hat{X})\subseteq\Null(\hat{Z})$, such that
\begin{equation*}
	A_1\bullet xx^H=c_1,\,\,\,A_{m_{\mathcal{F}}+1}\bullet  xx^H=c_{m_{\mathcal{F}}+1},\,\,\,A_{m_{\mathcal{F}}+2}\bullet  xx^H=c_{m_{\mathcal{F}}+2}.
\end{equation*} 
Furthermore, substituting $A_{m_{\mathcal{F}}+2}\bullet  xx^H=c_{m_{\mathcal{F}}+2}$ into (\ref{eq18}),  one obtains that $A_{m_{\mathcal{F}}}\bullet  xx^H=c_{m_{\mathcal{F}}}$.
Therefore, we find a rank-one matrix $xx^H$ that is feasible to $(SP_{m_{\mathcal{F}}+2})$ and satisfies the complementary condition \eqref{eq:PI(I.1)+hatXcomplem}, which means that  $xx^H$ is
 a rank-one optimal solution of $(SP_{m_{\mathcal{F}}+2})$. The proof is completed.
\end{proof}

\begin{remark}\label{remark:No-feasible-solution}  
Even under the Assumption \ref{ass1}, when the Property $\mathcal{I}$ holds, the original problem may indeed have no feasible solution. This fact can be illustrated by the following example.
\end{remark}

\begin{example}
	Consider the following instance of the model $(SP_3)$ over the real field:
	\begin{equation}\label{eq:example3.1SP}
		\begin{array}{lll}
		&\mbox{minimize}   &A_0\bullet X\\
			&\mbox{subject to}&A_1\bullet X=0,\\
			&&A_2\bullet X=0,\\
			& &A_3\bullet X\leq-2,\\
			&&X\succeq0,
		\end{array}
	\end{equation}
	where \begin{equation*}
		A_0=\left[\begin{array}{cc}
			0&0\\
			0&1
		\end{array}\right], 
		A_1=\left[\begin{array}{cc}
			1&0\\
			0&-1
		\end{array}\right], 
		A_2=\left[\begin{array}{cc}
			0&1\\
			1&0
		\end{array}\right],
		A_3=\left[\begin{array}{cc}
			-1&0\\
			0&-1
		\end{array}\right].
	\end{equation*}
	Its dual problem can be written as follows:
	\begin{equation}\label{eq:example3.1SD}
		\begin{array}{lll}
			&\mbox{maximize}   &2\mu_3\\
			&\mbox{subject to} & Z=A_0+\mu_1A_1+\mu_2A_2+\mu_3A_3\succeq0,\\
			&&  \mu_3\geq0,
		\end{array}
	\end{equation}
	Both problems \eqref{eq:example3.1SP} and \eqref{eq:example3.1SD} have interior feasible solutions
	$$\tilde{X}=\left[\begin{array}{cc}
		2&0\\
		0&2
	\end{array}\right] \text{ and } 
	\tilde{Z}=\left[\begin{array}{cc}
		0.25&0\\
		0&0.25
	\end{array}\right],\,
	(\tilde{\mu}_1,\,\tilde{\mu}_2,\,\tilde{\mu}_3)=(0.5,0,0.25),$$ 
	respectively. And their optimal solutions  are $\hat{X}=\left[\begin{array}{cc}
		1&0\\
		0&1
	\end{array}\right]$ and $(\hat{Z},\,\hat{\mu}_1,\,\hat{\mu}_2,\,\hat{\mu}_3)=(\textbf{0}_{2\times2},0.5,0,0.5)$ with an optimal objective value $1$, respectively. One can find a rank-one decomposition of $\hat{X}$, 
	$\hat{X}=\hat{x}_1\hat{x}_1^T+\hat{x}_2\hat{x}_2^T,$ such that 
	$$
	\begin{array}{l}
	\hat{x}_1=[\sqrt{2}/2,\sqrt{2}/2]^T,\,\quad\, \hat{x}_2=[-\sqrt{2}/2,\sqrt{2}/2]^T,\\
	A_1\bullet\hat{x}_1\hat{x}_1^T=A_1\bullet\hat{x}_2\hat{x}_2^T=0,\,\quad\,\hat{x}_1^TA_1\hat{x}_2\neq0,\\
	(A_2\bullet\hat{x}_1\hat{x}_1^T)(A_2\bullet\hat{x}_2\hat{x}_2^T)<0,
	\end{array}
	$$
	 that is the Property $\mathcal{I}$ holds for the pair $\hat{X}$ and $(\hat{Z},\hat{\mu}_1,\hat{\mu}_2,\hat{\mu}_3)$.
	However, their original problem 
	\begin{equation*}
		\begin{array}{lll}
			&\mbox{minimize}   &x^TA_0x\\
			&\mbox{subject to}&x^TA_1x=0,\\
			&&x^TA_2x=0,\\
			& &x^TA_3x\leq-2,\\
			&&x\in\mathcal{R}^2,
		\end{array}
	\end{equation*}
	has no feasible solution. 
\end{example}


%

\begin{corollary}
	Under Assumption \ref{ass1}, $(HQP_{m_{\mathcal{F}}+1})$ has an optimal solution and satisfies $v^*(HQP_{m_{\mathcal{F}}+1})$ $=v^*(SP_{m_{\mathcal{F}}+1})$. Moreover one can find an optimal solution to $(HQP_{m_{\mathcal{F}}+1})$ from any optimal solution pair of $(SP_{m_{\mathcal{F}}+1})$ and $(SD_{m_{\mathcal{F}}+1})$ in polynomial-time.  
\end{corollary}

\begin{proof}
 We may regard $(HQP_{m_{\mathcal{F}}+1})$ as an $(HQP_{m_{\mathcal{F}}+2})$ instance by adding an identical equation constraint $x^H0_{n\times n}x=0$.  Taking the identical equation as the first constraint, one can verify that the Property $\mathcal{I}$ (I.4.1)  fails for the instance, due to the cross-term $x_1^H0_{n\times n}x_2\equiv 0$ ($\forall x_1,x_2\in \mathcal{F}^n$). Thus by Theorem \ref{theorem:main-result}, one can find in polynomial-time an optimal solution to $(HQP_{m_{\mathcal{F}}+1})$ from its SDP relaxation, and $v^*(HQP_{m_{\mathcal{F}}+1})=v^*(HQP_{m_{\mathcal{F}}+2})=v^*(SP_{m_{\mathcal{F}}+2})=v^*(SP_{m_{\mathcal{F}}+1})$. 
\end{proof}

\section{S-lemma's extension on three real and four complex quadratic forms}
As an application of Theorem \ref{theorem:main-result}, in this section  we shall generalize S-lemma and Yuan's lemma to three real and four complex homogeneous quadratic functions. S-lemma \cite{Fradkov1973,Yakubo1971}  and Yuan's lemma  \cite{Yuan1990}, as you know, have a lot of important applications. However, they can not hold on three real homogeneous quadratic functions in general. 
Moreover, a sufficient condition can be found in \cite{Telarky2007}, which ensures that S-lemma holds  on three real homogeneous quadratic functions. One can easily verify that the sufficient condition is not a necessity indeed. What is an exact condition, that is a not only sufficient but also necessary condition, which can just guarantee the correctness of S-lemma on three real homogeneous quadratic functions? We manage to answer the question by using Theorem \ref{theorem:main-result}.

Let us consider two related optimization problems:
\begin{equation*}
	\begin{array}{lll}
		(HQP0_2)&\mbox{minimize}   &x^TA_0x\\
		&\mbox{subject to} &
		\begin{cases}
			x^TA_1x\leq0,\\
			x^TA_2x\leq0,
		\end{cases}
	\end{array}
\end{equation*}
and 
\begin{equation*}
	\begin{array}{lll}
		(HQP_3)&\mbox{minimize}   &x^TA_0x\\
		&\mbox{subject to} &
		\begin{cases}
			x^TA_1x\leq0,\\
			x^TA_2x\leq0,\\
			x^Tx\leq1.
		\end{cases}
	\end{array}
\end{equation*}
For $(HQP0_2)$ and $(HQP_3)$,  the following  facts are apparent: they both always have a feasible solution $x=0$;  the optimal objective value of $(HQP0_2)$ is either  $0$ or $-\infty$;  $(HQP_3)$ always has optimal solutions, and $v^*(HQP_3)=0$ if and only if $v^*(HQP0_2)=0$. The above facts tell us that one can characterize the optimal objective value of $(HQP0_2)$ by solving $(HQP_3)$. The SDP relaxation of $(HQP_3)$ is
\begin{equation*}
	\begin{array}{lll}
		(SP_3)&\mbox{minimize}   &A_0\bullet X\\
		&\mbox{subject to} &
		\begin{cases}
			A_1\bullet X\leq 0,\\
			A_2\bullet X\leq 0,\\
			I_n\bullet X\leq1,\\
			X\succeq0,
		\end{cases}
	\end{array}
\end{equation*}
and the dual problem of $(SP_3)$ is
\begin{equation*}
	\begin{array}{lll}
		(SD_3)&\mbox{maximize}   &-\mu_3\\
		&\mbox{subject to} & Z=A_0+\mu_1A_1+\mu_2A_2+\mu_3I_n\succeq0,\\
		&&  \mu_1\geq0,\,\mu_2\geq0,\,\mu_3\geq0,\\
	\end{array}
\end{equation*}
where $I_n$ is the $n\times n$ identity matrix. Apparently, if only $(HQP0_2)$ satisfies the Slater condition, then both  $(SP_3)$ and $(SD_3)$ satisfy the Slater condition, and Theorem \ref{theorem:main-result} can be applied  to  $(SP_3)$ and $(SD_3)$.

\begin{definition}
	Let $A_0,A_1$ and $A_2$ be three $n\times n$ real symmetric matrices. We say that Property  $\mathcal{IR}$ holds for $A_0,A_1,A_2$ if there exist three positive real numbers $\breve{\mu}_1,\breve{\mu}_2,\breve{\mu}_3$ and two vectors $\breve{x}_1,\breve{x}_2\in\mathcal{R}^n$, such that
	\begin{description}
		\item[(IR.1)] $\breve{Z}:=A_0+\breve{\mu}_1A_1+\breve{\mu}_2A_2+\breve{\mu}_3I_n\succeq0$ and $L(\breve{x}_1,\breve{x}_2)=\Null(\breve{Z})$, where $L(\breve{x}_1,\breve{x}_2)$ denotes the linear subspace spanned by the vectors $\breve{x}_1$ and $\breve{x}_2$;
		
		\item[(IR.2)] $\breve{x}_1^TA_1\breve{x}_1=\breve{x}_2^TA_1\breve{x}_2=0$ and  $\breve{x}_1^TA_1\breve{x}_2\neq0$;
		 
		\item[(IR.3)] $\breve{x}_1^TA_2\breve{x}_1<0<\breve{x}_2^TA_2\breve{x}_2$.
	\end{description}
\end{definition}
Now we give the first result of this section, which is an extension of S-lemma \cite{Yakubo1971} on three real homogeneous quadratic functions. And it improves Proposition 3.6 of \cite{Telarky2007} and Theorem 3.9 of \cite{Peng1997}. 
\begin{theorem}\label{S-lemma for three matrices}
	Let $A_0,A_1$ and $A_2$ be three $n\times n$ real symmetric matrices,  and let there be an $x_0\in\mathcal{R}^n$ such that $x_0^TA_1x_0<0$ and $x_0^TA_2x_0<0$. Then the following two statements are equivalent to each other.
	\begin{description}
		\item[{(i)}]  The system 
		\begin{equation}\label{system:A_0,A_1,A_2}
			\begin{cases}
				x^TA_0x<0,\\
				x^TA_1x\leq0,\\
				x^TA_2x\leq0,
			\end{cases}
		\end{equation}
		is not solvable and Property $\mathcal{IR}$ fails for $A_0,A_1,A_2$.
	
		\item[{(ii)}]
		There exist $\mu_1^0\geq0$ and $\mu_2^0\geq0$ such that $A_0+\mu_1^0A_1+\mu_2^0A_2\succeq0$.
	\end{description}
\end{theorem}


\begin{proof}
	As $(HQP0_2)$ satisfies the Slater condition, then both  $(SP_3)$ and $(SD_3)$ satisfy the Slater condition and have optimal solutions. Let $\hat{X}$ and $(\hat{Z},\hat{\mu}_1,\hat{\mu}_2,\hat{\mu}_3)$ be any optimal solution pair to $(SP_3)$ and $(SD_3)$.
	
{\noindent\bf\Large  ``(i)$\Longrightarrow$(ii)".}\\
	 Since the system (\ref{system:A_0,A_1,A_2}) has no solution,  $x=0$  must be an optimal solution of $(HQP0_2)$, that is $v^*(HQP_3)=v^*(HQP0_2)=0.$ 
 Due to Property $\mathcal{IR}$ fails for $A_0,A_1,A_2$, the pair of $\hat{X}$ and $(\hat{Z},\hat{\mu}_1,\hat{\mu}_2,\hat{\mu}_3)$ does not satisfy Property $\mathcal{I}$. From Theorem \ref{theorem:main-result}, one has  
 $$
 -\hat{\mu}_3=v^*(SD_3)=v^*(SP_3)=v^*(HQP_3)=0,
 $$  
 which implies that $\hat{Z}=A_0+\hat{\mu}_1A_1+\hat{\mu}_2A_2\succeq0$ and then the statement (ii) is true.
	
{\noindent\bf\Large ``(ii)$\Longrightarrow$(i)".}\\
	Put $Z_0=A_0+\mu_1^0A_1+\mu_2^0A_2$ and $\mu_3^0=0$. Then  $(Z_0,\mu_1^0,\mu_2^0,\mu_3^0)$ is a feasible solution to $(SD_3)$. Therefore one has
	$$0=\mu_3^0\leq v^*(SD_3)=v^*(SP_3)\leq v^*(HQP_3)\leq0\,\Longrightarrow\,v^*(SP_3)= v^*(HQP_3)= v^*(HQP0_2)=0,$$
	which implies that the system (\ref{system:A_0,A_1,A_2}) has no solution and the pair of $\hat{X}$ and $(\hat{Z},\hat{\mu}_1,\hat{\mu}_2,\hat{\mu}_3)$ does not satisfy the Property $\mathcal{I}$. To complete the proof, we need only to show the Property $\mathcal{IR}$ fails for $A_0,A_1,A_2$. In fact, if the Property $\mathcal{IR}$ holds for $A_0,A_1,A_2$,  
	one can choose two appropriate positive real numbers $\alpha$ and $\beta$ such that
	\begin{equation}\label{eq:breveX-1}
		\breve{X}:=\alpha\breve{x}_1\breve{x}_1^T+\beta\breve{x}_2\breve{x}_2^T\succeq0,\,\,A_1\bullet\breve{X}=A_2\bullet\breve{X}=0,\,\,  I_n\bullet	\breve{X}=1.
	\end{equation}
	Thus $\breve{X}$ is a feasible solution to $(SP_3)$, and its objective value is
	\begin{equation}\label{eq:breveX-2}
		A_0\bullet\breve{X}=A_0\bullet\breve{X}-\breve{Z}\bullet\breve{X}=-\breve{\mu}_1A_1\bullet\breve{X}-\breve{\mu}_2A_2\bullet\breve{X}-\breve{\mu}_3I_n\bullet\breve{X}=-\breve{\mu}_3<0,
	\end{equation}
	which contradicts with $v^*(SP_3)=0$. Therefore the statement (i) is true.
\end{proof}

The following theorem can be regarded as an extension of  Yuan's lemma \cite{Yuan1990} on three real quadratic forms, in which the added condition is exacter and weaker than that in Theorem 5 of \cite{Yuan1999}.
\begin{theorem}\label{Yuan's lemma extend}
	 Let $A_0,A_1$ and $A_2$ be three $n\times n$ real symmetric matrices. Then the following two statements are equivalent to each other.
	\begin{description}
		\item[{(i)}] {\bf (i.1)} $max\{x^TA_0x,x^TA_1x,x^TA_2x\}\geq0\,\forall\,x\in\mathcal{R}^n.\,$ {\bf (i.2)} There exists a permutation $\{j_0,j_1,j_2\}$ of $\{0,1,2\}$ such that, the system 
			\begin{equation}\label{system:A_i0,A_i1,A_i2=0}
				\begin{cases}
					x^TA_{j_0}x<0,\\
					x^TA_{j_1}x=0,\\
					x^TA_{j_2}x=0,
				\end{cases}
			\end{equation}
			is not solvable and Property $\mathcal{IR}$ fails for $A_{j_0},A_{j_1},A_{j_2}.$ 
			
		\item[{(ii)}]
		There exist $\mu^0_k\geq0$ $(k=0,1,2)$ and $\sum\limits_{k=0}^{2}\mu^0_k=1$ such that $\mu^0_0A_0+\mu^0_1A_1+\mu^0_2A_2\succeq0$.
	\end{description}
\end{theorem}
\begin{proof}
	{\bf\Large ``(ii)$\Longrightarrow$(i)".}\\
	As the sum of three nonnegative numbers $\mu^0_0,\mu^0_1,\mu^0_2$ is equal to $1$,  without loss of generality, we assume $\mu^0_0>0$. Due to $\mu^0_0>0$ and $\mu^0_0A_0+\mu^0_1A_1+\mu^0_2A_2\succeq0$,  one can easily obtain that 
	$$
	max\{x^TA_0x,x^TA_1x,x^TA_2x\}\geq0\,\forall\,x\in\mathcal{R}^n
	$$ 
	and the system 
	\begin{equation*}\label{system:A_0,A_1,A_2=0}
			\begin{cases}
				x^TA_{0}x<0,\\
				x^TA_{1}x=0,\\
				x^TA_{2}x=0,
			\end{cases}
	\end{equation*}
	is not solvable. Then we need only to prove Property $\mathcal{IR}$ fails for $A_{0},A_{1},A_{2}$.  In fact, if Property $\mathcal{IR}$ holds for $A_{0},A_{1},A_{2}$,  then following the proof of Theorem \ref{S-lemma for three matrices}, one can obtain a positive semidefinite matrix $\breve{X}$ that satisfies \eqref{eq:breveX-1} and \eqref{eq:breveX-2}. Then the inner product of  $\breve{X}$ and $\mu^0_0A_0+\mu^0_1A_1+\mu^0_2A_2$ becomes
		$$
		0\leq(\mu^0_0A_0+\mu^0_1A_1+\mu^0_2A_2)\bullet \breve{X}=\mu^0_0A_0 \bullet \breve{X}=-\mu^0_0\breve\mu_3<0,
		$$
		which is a contradiction. Therefore the statement (i) is true.

	\noindent``{\bf\Large (i)$\Longrightarrow$(ii)}".\\
	Without loss of generality, we assume $(j_0,j_1,j_2)=(0,1,2)$. If $max\{x^TA_1x,x^TA_2x\}\geq0$, by Yuan's lemma (Lemma 2.3 of \cite{Yuan1990}), the statement (ii) is obviously true. So we assume that there exists a vector $x_0\in\mathcal{R}^n$ that satisfies 
	\begin{equation}\label{eq:xzeroA1A2}
		x_0^TA_1x_0<0 \text{ and } x_0^TA_2x_0<0.
	\end{equation}
	Then we assert that the system (\ref{system:A_0,A_1,A_2}) is not solvable because, if some vector $ x_1\in\mathcal{R}^n$ is feasible for the system (\ref{system:A_0,A_1,A_2}), then  $x_1$ must satisfy
	\begin{equation}\label{eq:xoneA1A2}
		x_1^TA_0x_1<0,\, x_1^TA_1x_1<0,\, x_1^TA_2x_1=0\,\, \text{ or }\,x_1^TA_0x_1<0,\ x_1^TA_1x_1=0,\ x_1^TA_2x_1<0,
	\end{equation}
	due to $max\{x^TA_0x,x^TA_1x,x^TA_2x\}\geq0$ and the system (\ref{system:A_i0,A_i1,A_i2=0}) is not solvable. Let $\epsilon$ be a sufficiently small positive real number. Combining \eqref{eq:xoneA1A2} and \eqref{eq:xzeroA1A2}, one obtains either
	$$
	(x_1+\epsilon x_0)^TA_k(x_1+\epsilon x_0)<0\,\, (\forall k=0,1,2) \text{ or } (x_1-\epsilon x_0)^TA_k(x_1-\epsilon x_0)<0\,\, (\forall k=0,1,2), 
	$$
	which contradicts with $max\{x^TA_0x,x^TA_1x,x^TA_2x\}\geq0\,\forall\,x\in\mathcal{R}^n$. Hence the system (\ref{system:A_0,A_1,A_2}) is not solvable. Then by Theorem \ref{S-lemma for three matrices}, there must exist $\mu^0_1\geq0$ and $\mu^0_2\geq0$ such that $A_0+\mu^0_1A_1+\mu^0_2A_2\succeq0$, that is the statement (ii) is true.		
\end{proof}

Similarly, by using Theorem \ref{theorem:main-result}, on can generalize S-lemma and Yuan's lemma to four complex quadratic forms. Since their proof processes are very similar to those of the real-valued results, we state only the corresponding results but omit their proof processes. 

\begin{definition}
	Let $A_0,A_1,A_2,A_3$ be four $n\times n$ Hermitian matrices. We say that Property $\mathcal{IC}$ holds for $A_0,A_1,A_2,A_3$ if there exist four positive real numbers $\breve{\mu}_1,\breve{\mu}_2,\breve{\mu}_3,\breve{\mu}_4$ and two vectors $\breve{x}_1,\breve{x}_2\in\mathcal{C}^n$, such that
	\begin{description}
		\item[(IC.1)] $\breve{Z}:=A_0+\breve{\mu}_1A_1+\breve{\mu}_2A_2+\breve{\mu}_3A_3+\breve{\mu}_4I_n\succeq0$ and $L(\breve{x}_1,\breve{x}_2)=\Null(\breve{Z})$;
		
		\item[(IC.2)] $\breve{x}_1^HA_1\breve{x}_1=\breve{x}_2^HA_1\breve{x}_2=0$, $\re(\breve{x}_1^HA_1\breve{x}_2)\neq0$;
		
		\item[(IC.3)] $\breve{x}_1^HA_2\breve{x}_1=\breve{x}_2^HA_2\breve{x}_2=0$, $\re(\breve{x}_1^HA_2\breve{x}_2)=0$, $\im(\breve{x}_1^HA_2\breve{x}_2)\neq0$; 
		
		\item[(IC.4)] $\breve{x}_1^HA_3\breve{x}_1<0< \breve{x}_2^HA_3\breve{x}_2$.
	\end{description}
\end{definition}
\begin{theorem}
	Let $A_0,A_1,A_2,A_3$ be four $n\times n$ Hermitian matrices, and let there be one vector $x_0\in\mathcal{C}^n$ such that $x_0^HA_kx_0<0$ for all $k=1,2,3$. Then the following two statements are equivalent to each other.
	\begin{description}
		\item[{(i)}]  The system 
		\begin{equation}
			\begin{cases}
				x^HA_0x<0,\\
				x^HA_1x\leq0,\\
				x^HA_2x\leq0,\\
				x^HA_3x\leq0,
			\end{cases}
		\end{equation}
		is not solvable, and Property $\mathcal{IC}$ fails for $A_0,A_1,A_2,A_3$.
		\item[{(ii)}]
		There exist $\mu^0_k\geq0$ $(k=1,2,3)$ such that $A_0+\mu^0_1A_1+\mu^0_2A_2+\mu^0_3A_3\succeq0$.
	\end{description}
\end{theorem}
\begin{theorem}
	Let $A_0,A_1,A_2,A_3$ be four $n\times n$ Hermitian matrices. Then the following two statements are equivalent to each other.
	\begin{description}
		\item[{(i)}] {\bf (i.1)} $max\{x^HA_0x,x^HA_1x,x^HA_2x,x^HA_3x\}\geq0\,\forall\,x\in\mathcal{C}^n.$\, {\bf (i.2)}  There exists a permutation $\{j_0,j_1,j_2,j_3\}$ of $\{0,1,2,3\}$ such that the system 
		\begin{equation}
			\begin{cases}
				x^HA_{j_0}x<0,\\
				x^HA_{j_1}x=0,\\
				x^HA_{j_2}x=0,\\
				x^HA_{j_3}x=0,
			\end{cases}
		\end{equation}
		is not solvable and Property $\mathcal{IC}$ fails for $A_{j_0},A_{j_1},A_{j_2},A_{j_3}.$
		
		\item[{(ii)}]
		There exist $\mu^0_k\geq0$ $(k=0,1,2,3)$ and $\sum\limits_{k=0}^{3}\mu^0_k=1$ such that $\mu^0_0A_0+\mu^0_1A_1+\mu^0_2A_2+\mu^0_3A_3\succeq0$.
	\end{description}
\end{theorem}

\section{Testing Property $\mathcal{I}$ numerically}\label{sect:num}
In this section, we shall perform numerical experiments. Consider the following model over the complex number field:
\begin{equation*}
	\begin{array}{lll}
		(HQP_4)&\mbox{minimize}   &x^HA_0x\\
		&\mbox{subject to} &x^HA_1x\leq c_1,\\
		& &x^HA_2x\leq c_2,\\
		& &x^HA_3x\leq c_3,\\
		& &x^HA_4x\leq c_4\,(c_4\neq0),\\
	\end{array}
\end{equation*}
where $A_0,\,A_i$ are $n\times n$ complex Hermitian matrices and $c_i\in\mathcal{R}$, $i=1,2,3,4$.
As SDP solvers can only return approximate optimal solutions within a tolerance, we can only verify whether $(HQP_4)$ problem satisfies Property $\mathcal{I}$ within an error. The specific operations are as follows:

Let $\hat{X}$ and $(\hat{Z},\hat{\mu}_1,\hat{\mu}_2,\hat{\mu}_3,\hat{\mu}_4)$ be a pair of approximate optimal solutions to $(SP_4)$ and $(SD_4)$ respectively, which are returned from an SDP solver within an error $\varepsilon_1$. 
 Then,  we perform an eigenvalue decomposition for $\hat{X}$ and $\hat{Z}$:
\begin{equation*}
	\hat{X}=Q_1^H\Lambda_1Q_1 \text{ and } \hat{Z}=Q_2^H\Lambda_2Q_2,
\end{equation*}
where $Q_i$ is orthonormal and $\Lambda_i=diag(\lambda_{i1},\lambda_{i2},\cdots,\lambda_{in})$, with $\lambda_{ij}\geq0$, $j=1,2,\cdots,n$; $i=1,2$. Put
\begin{equation*}
	\hat{\lambda}_{ij}:=\begin{cases}
		\lambda_{ij}, \qquad&\text{ if } \lambda_{ij}>\varepsilon_2,\\
		0, \qquad&\text{ if } \lambda_{ij}\leq\varepsilon_2,
	\end{cases}
\text{ for all } j=1,2,\cdots,n; i=1,2.
\end{equation*}
 Let us purify the solutions by using 
\begin{equation*}
	X^*:=Q_1^Hdiag(\hat{\lambda}_{11},\hat{\lambda}_{12},\cdots,\hat{\lambda}_{1n})Q_1\text{ and }Z^*:=Q_2^Hdiag(\hat{\lambda}_{21},\hat{\lambda}_{22},\cdots,\hat{\lambda}_{2n})Q_2
\end{equation*}
instead of $\hat{X}$ and $\hat{Z}$, while keeping $\mu_i^*=\hat{\mu}_i$, $i=1,2,3,4$. We call $X^*$ and $(Z^*,\mu_1^*,\mu_2^*,\mu_3^*,\mu_4^*)$ to be a pair of purified $(\varepsilon_1,\varepsilon_2)$-approximate optimal solutions. Then we redefine the Property $\mathcal{I}$ in the numerical sense.
\begin{definition}
	For $X^*$ and $(Z^*,\mu_1^*,\mu_2^*,\mu_3^*,\mu_4^*)$, a given pair of optimal solutions for $(SP_4)$ and $(SD_4)$ respectively, we say that this pair has Property $\mathcal{I}(\varepsilon_2)$ if the following conditions are simultaneously satisfied:
	\begin{description}
		\item[(I.1)]$\mu_i^*>\varepsilon_2$, $i=1,2,3,4$;

		\item[(I.2)]  $\rank(Z^*,\varepsilon_2)=n-2$;
		
		\item[(I.3)]  $\rank(X^*,\varepsilon_2)=2$;
		
		\item[(I.4)]  there is a rank-one decomposition of $X^*$, $X^*=x_1^*{x_1^*}^H+x_2^*{x_2^*}^H$, such that 
		\begin{description}
			\item[(I.4.1)]
			$\left|\left(A_1-\dfrac{c_1}{c_4}A_4\right)\bullet x_1^*{x_1^*}^H\right|\leq\varepsilon_2,\, \left|\left(A_1-\dfrac{c_1}{c_4}A_4\right)\bullet x_2^*{x_2^*}^H\right|\leq\varepsilon_2,$
			
			$\quad\,\,\left|\re\left( {x_1^*}^H\left(A_1-\dfrac{c_1}{c_4}A_4\right)x_2^*\right)\right|>\varepsilon_2;$
			
			\item[(I.4.2)] 
			$\left|\left(A_2-\dfrac{c_2}{c_4}A_4\right)\bullet x_1^*{x_1^*}^H\right|\leq\varepsilon_2,\, \left|\left(A_2-\dfrac{c_2}{c_4}A_4\right)\bullet x_2^*{x_2^*}^H\right|\leq\varepsilon_2,$
			
			$\quad\,\,\left|\re\left( {x_1^*}^H\left(A_2-\dfrac{c_2}{c_4}A_4\right)x_2^*\right)\right|\leq\varepsilon_2,$
			
			$\quad\,\,\left|\im\left( {x_1^*}^H\left(A_2-\dfrac{c_2}{c_4}A_4\right)x_2^*\right)\right|>\varepsilon_2;$
			
			\item[(I.4.3)] 
			$\left(\left(A_3-\dfrac{c_3}{c_4}A_4\right)\bullet x_1^*{x_1^*}^H\right)\left(\left(A_3-\dfrac{c_3}{c_4}
			A_4\right)\bullet x_2^*{x_2^*}^H\right)<-\varepsilon_2^2.$
		\end{description}
	\end{description}
\end{definition}
We shall use MATLAB and CVX for numerical experiments. Throughout our tests, let $\varepsilon_1$ be the default precision of CVX and $\varepsilon_2=1.0e-04$. For any given positive integer $n$, we generate the matrices $A_i$ $(i=0,1,2,3)$ in $\mathcal{C}^{n\times n}$, for which the real and imaginary parts of all the entries are uniformly distributed in $[-10,10]$. Then redefine 
$$
A_i:=(A_i+A_i^H)/2,\ i=0,1,2,3.
$$ 
Furthermore, in order to ensure that Assumption \ref{ass1} (ii) is met,  we need to generate a special $A_4$. Our approach is to randomly generate a positive definite matrix $Z$, for which the real and imaginary parts of all the entries are uniformly distributed in $[-40,40]$. Then denote $A_4:=Z-A_0-A_1-A_2-A_3.$ Finally, we randomly generate $c_i$ ($i=1,2,3,4$) such that Assumption \ref{ass1} (i) holds.

By the above method, we generate $1000$ instances about the model $(HQP_4)$  for all the dimension numbers $n=2,3,\cdots,10$. And the numbers of those instances, whose Property $\mathcal{I}(\varepsilon_2)$ fails, are presented in Tables \ref{tab:1}. From Tables \ref{tab:1}, one can observe that, most of the instances violates the Property $\mathcal{I}(\varepsilon_2)$, and so optimal solutions can be found to the corresponding original problems,  which highlights the numerical effectiveness of Theorem \ref{theorem:main-result}.

\begin{table}[htbp]
	\centering
	\caption{ The numbers $k$ of those instances whose Property $\mathcal{I}(\varepsilon_2)$ fails}
	\label{tab:1}
	\begin{threeparttable}
		\begin{tabular}{l|lllllllll}
			\noalign{\smallskip}
			\toprule
			$n$ &2&3&4&5&6&7&8&9&10 \\
			\midrule
			$k$ &958&944&932&909&933&946&944&932&950\\
			\bottomrule
		\end{tabular}
	\end{threeparttable}
\end{table}

Finally, we show an instance generated by our method, which has a positive optimal-value gap between the original problem and its SDP relaxation.
\begin{example}
Let
	$$A_0=\left[\begin{array}{cc}
		-7&-4\boldsymbol{i}\\
		4\boldsymbol{i}&-6
	\end{array}\right], A_1=\left[\begin{array}{cc}
		9&8-10\boldsymbol{i}\\
		8+10\boldsymbol{i}&18
	\end{array}\right], A_2=\left[\begin{array}{cc}
		6&7+3\boldsymbol{i}\\
		7-3\boldsymbol{i}&0
	\end{array}\right],$$ $$A_3=\left[\begin{array}{cc}
		3&-5+6\boldsymbol{i}\\
		-5-6\boldsymbol{i}&7
	\end{array}\right], A_4=\left[\begin{array}{cc}
		7&4\boldsymbol{i}\\
		-4\boldsymbol{i}&-8
	\end{array}\right], c_1=6,\, c_2=7,\, c_3=9,\, c_4=4.$$
 For the corresponding SDP relaxation, the purified $(\varepsilon_1,\varepsilon_2)$-approximate optimal solutions are as follows:\\
	\begin{equation*}
		\begin{aligned}
			X^*&\approx\left[\begin{array}{lc}
				0.5939&0.0700+0.4095\boldsymbol{i}\\
				0.0700-0.4095\boldsymbol{i}&0.4292
			\end{array}\right],\\
			Z^*&\approx \left[\begin{array}{cc}
				0&0\\
				0&0
			\end{array}\right],\\
			\mu_1^*&\approx0.1821,\quad \mu_2^*\approx0.2708,\quad \mu_3^*\approx0.6705,\quad \mu_4^*\approx0.2464.
		\end{aligned}
	\end{equation*}
 And $v^*(SP_4)= v^*(SD_4)\approx-10.0083$, $v^*(HQP_4)\approx-9.6974$ attained at  $x^*\approx[-0.1308-0.7089\boldsymbol{i},-0.6379]^T$.  One can verify that it has Property $\mathcal{I}(\varepsilon_2)$ indeed.  Firstly, $\mu_1^*\mu_2^*\mu_3^*\mu_4^*\neq0$, $\rank(Z^*,\varepsilon_2)=0$ and $\rank(X^*,\varepsilon_2)=2$. Secondly, $X^*$ has a rank-one decomposition, $X^*=x_1^*{x_1^*}^H+x_2^*{x_2^*}^H$, such that
	\begin{equation*}
		x_1^*\approx\left[\begin{array}{c}
			0.4256-0.1259\boldsymbol{i}\\
			-0.3083-0.4456\boldsymbol{i}
		\end{array}\right]\text{ and } x_2^*\approx\left[\begin{array}{c}
			0.3875-0.4968\boldsymbol{i}\\
			-0.0849-0.3583\boldsymbol{i}
		\end{array}\right];
	\end{equation*}
	$$\left|\left(A_1-\dfrac{c_1}{c_4}A_4\right)\bullet x_1^*{x_1^*}^H\right|\approx 1.7944e-10\leq\varepsilon_2,\, \left|\left(A_1-\dfrac{c_1}{c_4}A_4\right)\bullet x_2^*{x_2^*}^H\right|\approx 1.7945e-10\leq\varepsilon_2,$$
	$$\left|\re\left( {x_1^*}^H\left(A_1-\dfrac{c_1}{c_4}A_4\right)x_2^*\right)\right|\approx 1.7032>\varepsilon_2;$$
	$$\left|\left(A_2-\dfrac{c_2}{c_4}A_4\right)\bullet x_1^*{x_1^*}^H\right|\approx 1.5521e-09\leq\varepsilon_2,\, \left|\left(A_2-\dfrac{c_2}{c_4}A_4\right)\bullet x_2^*{x_2^*}^H\right|\approx 1.5521e-09\leq\varepsilon_2,$$
	$$\left|\re\left( {x_1^*}^H\left(A_2-\dfrac{c_2}{c_4}A_4\right)x_2^*\right)\right|\approx 4.4409e-16\leq\varepsilon_2,\,
	\left|\im\left( {x_1^*}^H\left(A_2-\dfrac{c_2}{c_4}A_4\right)x_2^*\right)\right|\approx 3.5430>\varepsilon_2;$$
	$$\left(\left(A_3-\dfrac{c_3}{c_4}A_4\right)\bullet x_1^*{x_1^*}^H\right)\left(\left(A_3-\dfrac{c_3}{c_4}
	A_4\right)\bullet x_2^*{x_2^*}^H\right)\approx -17.7092<-\varepsilon_2^2.$$
\end{example}

\section*{Declarations}
{\bf Conflict of interest} The authors declare that they have no conflict of interest.
\section*{Remark}
	This manuscript was submitted to Mathematical Programming on February 28, 2023.
	\begin{figure}[htbp]
		\centering
		\includegraphics[width=12cm,height=15cm]{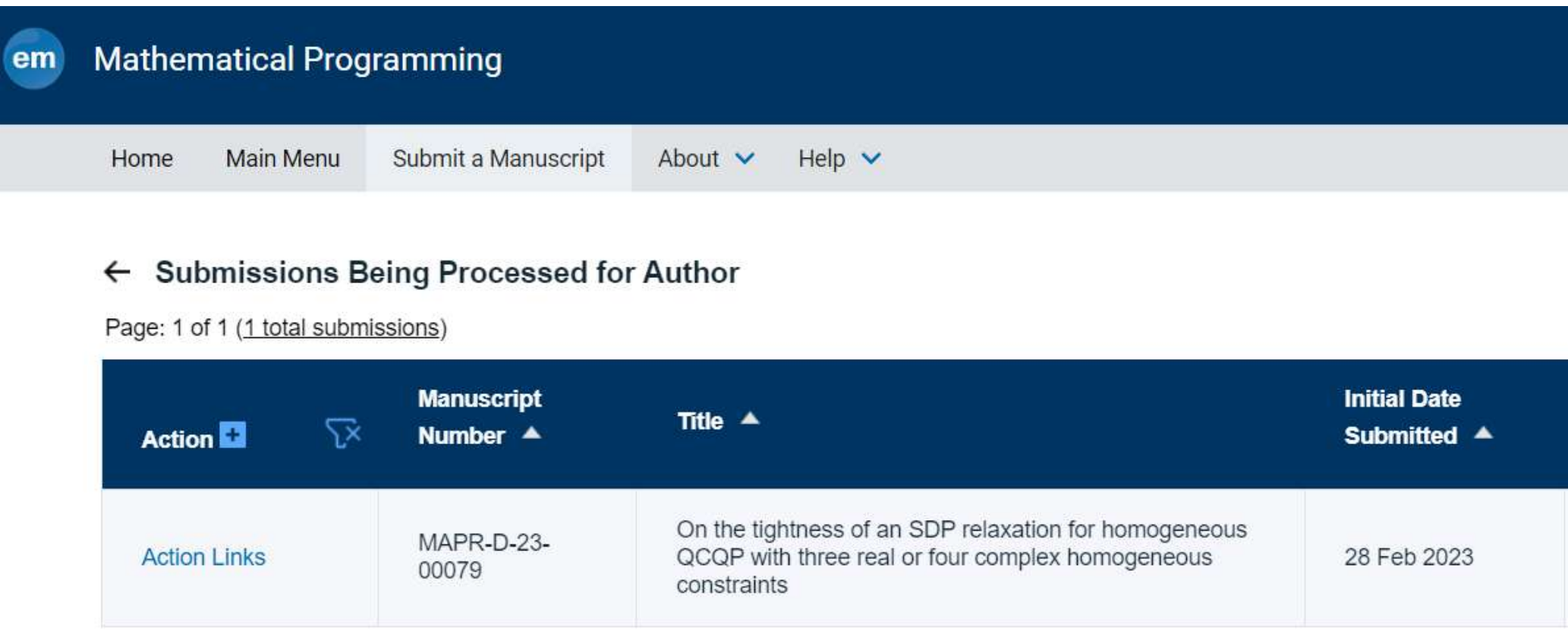}
		\label{fig:1}
	\end{figure}

\end{document}